\documentclass[12pt,reqno]{amsart}
\usepackage{fullpage}
\usepackage[top=1.5in, bottom=2in, left=2.5in, right=2.5in]{geometry}
\usepackage{graphicx,amsmath,amssymb,amsfonts,amsthm,enumitem,bm,xcolor}
\usepackage{fullpage}
\usepackage{cleveref}

\usepackage[normalem]{ulem}

\usepackage{url}

\usepackage{color}

\renewcommand{\d}{{\delta}}

\renewcommand{\o}{{\omega}}

\renewcommand{\(}{\left\(}
\renewcommand{\)}{\right\)}

\newcommand{\pa}[2]{\left(\frac{#1}{#2}\right)}

\numberwithin{equation}{section}
\theoremstyle{plain}
\newtheorem{theorem}{Theorem}[section]
\newtheorem{lemma}[theorem]{Lemma}

\newtheorem{conjecture}[theorem]{Conjecture}
\newtheorem*{conjecture*}{Conjecture}

\newtheorem{proposition}[theorem]{Proposition}

\theoremstyle{definition}

\newtheorem*{remark}{Remark}

\numberwithin{equation}{section}

\renewcommand{\pmod}[1]{\ \left( \mathrm{mod} \, #1 \right)}

\setlist[enumerate]{leftmargin=*,label=\rm{(\arabic*)}}

\makeatletter
\@namedef{subjclassname@2020}{%
	\textup{2020} Mathematics Subject Classification}
\makeatother

\allowdisplaybreaks

\title{On congruence conjectures of Andrews and Bachraoui}
\author{Koustav Banerjee}
\author{Kathrin Bringmann}
\address{University of Cologne, Department of Mathematics and Computer Science, Weyertal 86-90, 50931 Cologne, Germany}
\email{kbanerj1@uni-koeln.de}
\email{kbringma@uni-koeln.de}
\author{Mohamed El Bachraoui}
\address{Dept. Math. Sci, United Arab Emirates University, PO Box 15551, Al-Ain, UAE}
\email{melbachraoui@uaeu.ac.ae}

\makeatletter
\@namedef{subjclassname@2020}{%
	\textup{2020} Mathematics Subject Classification}
\makeatother


\begin{document}

\begin{abstract}

Andrews and the third author recently studied congruences for certain restricted two-color partitions.  They made two conjectures for Ramanujan-type congruences and a vanishing identity for the limiting sequence. In this paper, we settle these conjectures by relating the corresponding generating function to modular forms and mock theta functions. 

\end{abstract}

\maketitle




\section{Introduction and Statements of results} 

 In \cite[(1.1)]{AB}, Andrews and the third author considered restricted two-color partitions (see \cite[Definition 1]{AB}) $C(k,n)$ (for $k\in \mathbb{N}$), which count the number of two-color partitions of $n$, where the smallest part is odd and occurs at least once in blue color, every even part in blue color is at least $2k-1$ greater than the smallest odd part, and even parts of same color are distinct. They showed that
\[
C_k(q):=\sum_{n\ge 0}C\left(k,n\right)q^n=\sum_{n\ge 0}\frac{\left(-q^{2n+2k},-q^{2n+2};q^2\right)_{\infty}q^{2n+1}}{\left(q^{2n+1};q^2\right)^2_{\infty}},
\]  
where $(a)_n=(a;q)_n:=\prod_{j=0}^{n-1}(1-aq^j)$ ($n\in \mathbb{N}_0\cup \{\infty\}$). Letting $k\to \infty$, in \cite[Section 10]{AB} they introduced
\[
C(q):=\lim_{k\to \infty}C_k(q)=\sum_{n\ge 0}\frac{\left(-q^{2n+2};q^2\right)_{\infty}q^{2n+1}}{\left(q^{2n+1};q^2\right)^2_{\infty}}=:\sum_{n\ge 0}c(n)q^n.
\]
In \cite[Remark 3]{AB}, the authors then represented $C(q)$ as 
\begin{equation}\label{C}
C(q)=qA(q)S(q),
\end{equation}
where\footnote{The definition of $A(q)$ slightly differs (by a prefactor $q$) from Andrews and Bachraoui's definition of $A(q)$, see \cite[Remark 3]{AB}.}  
\begin{align}\label{A}
A(q)&:=\!\frac{\left(-q^2;q^2\right)_{\infty}}{\left(q;q^2\right)^2_{\infty}}\!=:\!\sum_{n\ge 0}a(n)q^n,\\\label{S}
 S(q)&:=\!\sum_{n\ge 0}\frac{\left(q;q^2\right)^2_n q^{2n}}{\left(-q^2;q^2\right)_n}\!=:\!\sum_{n\ge 0}s(n)q^n.
\end{align}
In the context of $S(q)$, they \cite[Remark 2]{AB} remarked that
\begin{quote}
	{\it It is therefore natural to expect modular or mock-modular behaviour of $S(q)$.}	
\end{quote}
 According to Ramanujan \cite{BB}, a {\it mock theta function}\footnote{In modern language, a mock theta function is the holomorphic part of a so-called harmonic Maass form (a certain non-holomorphic modular object).} is a function $F$ that has asymptotic expansions at every roots of unity of the same type as those of weakly holomorphic modular forms, however there is not a single weakly holomorphic modular form whose asymptotic expansions agree at all roots of unity with that of $F$. 

In this paper, we confirm this by relating $S(q)$ to mock theta functions. Recall Ramanujan's third order mock theta function $\o(q)$ 
\begin{equation}\label{Ramanujanomega}
\omega(q):=\sum_{n\ge0}\frac{q^{2n(n+1)}}{\left(q;q^2\right)^2_{n+1}}=:\sum_{n\ge 0}c_{\o}(n)q^n
\end{equation}
and  McIntosh's \cite{McIntosh} second order mock theta function $B(q)$ 
\begin{equation}\label{McIntosh}
B(q):=\sum_{n\ge 0}\frac{\left(-q^2;q^2\right)_n q^{n(n+1)}}{\left(q;q^2\right)^2_{n+1}}=:\sum_{n\ge 0}c_B(n)q^n.
\end{equation}
\begin{theorem}\label{lem1}
	We have
	\[
	S(q)=2B(-q)-\frac{\left(q;q^2\right)^2_{\infty}}{\left(-q^2;q^2\right)_{\infty}}\o(-q).
	\]
\end{theorem}
 Similar to the Ramanujan’s congruences for the partition function
\begin{align*}
p(5n+4)&\equiv 0\pmod{5},\\ 
p(7n+5)&\equiv 0\pmod{7},\\ 
p(11n+6)&\equiv 0\pmod{11}.
\end{align*}
 Andrews and the third author conjectured the following \cite[Conjectures 1 and 2]{AB}.
\begin{conjecture}\label{conj1a}
For $n\in \mathbb{N}_0$, we have
\begin{align*}
c(8n+4)\equiv 0\pmod{4}.
\end{align*}
\end{conjecture}
\begin{conjecture}\label{conj1b}
	For $n\in \mathbb{N}_0$, we have
	\begin{align*}
	c(8n+6)\equiv 0\pmod{8}.
	\end{align*}
\end{conjecture}

In this paper, we prove these conjectures.

\begin{theorem}\label{ABConj1}
\Cref{conj1a} is true.
\end{theorem}

\begin{theorem}\label{ABConj1a}
	\Cref{conj1b} is true.
\end{theorem}

 Moreover, we find and prove the following congruence for $c(n)$. 

\begin{theorem}\label{newthm}
For $n\in \mathbb{N}_0$, we have
\[
	c(16n+13)\equiv 0 \pmod 4.
\]
\end{theorem}

In context of $S(q)$, Andrews and the third author proposed \cite[Conjecture 3]{AB} the following conjecture: 
\begin{conjecture}\label{Conj2}
For $n\in \mathbb{N}_0$, we have 
\[s(4n+1)=0.\]
\end{conjecture}

In this paper, we confirm this conjecture.

\begin{theorem}\label{ABConj2}
	\Cref{Conj2} is true.
\end{theorem}




The remainder of the paper is organized as follows. In \Cref{sec1} we recall some well-known $q$-series transformations and identities related to modular forms and mock theta functions. In \Cref{sec2} we relate $S(q)$ to mock theta functions. In \Cref{sec3}, \Cref{sec3a}, and \Cref{sec3b}, we prove \Cref{ABConj1}, \Cref{ABConj1a}, and \Cref{newthm}, respectively. In \Cref{sec4} we prove \Cref{ABConj2}. Finally in \Cref{sec5} we give few follow up questions for future research.


\section{Preliminaries}\label{sec1}

In this section, we recall $q$-series identities and congruences for theta functions and mock theta functions.

We start with the following identity due to Ramanujan \cite[equation (3.17)$_\text{R}$]{A0}.


\begin{lemma}\label{lemident1} 
	We have 
	\begin{align*}
	&\sum_{n\ge 0}\!\frac{\left(-aq,-bq\right)_nq^{n+1}}{\left(-cq\right)_n}\!
	=\!\frac{c}{ab}\!\sum_{n\ge 1}\!\frac{\left(-c^{-1}\right)_n \pa{ab}{c}^n\! q^{\frac{n(n+1)}{2}}}{\left(\frac{aq}{c},\frac{bq}{c}\right)_n}\\
	&\hspace{7 cm}-\!\frac{c\left(-aq,-bq\right)_{\infty}}{ab\left(-cq\right)_{\infty}}\!\sum_{n\ge 1}\!\frac{\pa{ab}{c^2}^nq^{n^2}}{\left(\frac{aq}{c},\frac{bq}{c}\right)_n}.
	\end{align*}
\end{lemma}

Recall the {\it theta function} 
\begin{equation}\label{theta}
\Theta(q):=\sum_{n\in \mathbb{Z}}q^{n^2}=\frac{\left(q^2;q^2\right)^5_{\infty}}{(q)^2_{\infty}\left(q^4;q^4\right)^2_{\infty}}.
\end{equation}
Note that
\begin{equation}\label{thetas}
\Theta(-q)=\sum_{n\in \mathbb{Z}}(-1)^nq^{n^2}=\frac{(q)^2_{\infty}}{\left(q^2;q^2\right)_{\infty}}.
\end{equation}
Moreover, we require the ``odd" theta function, defined by 
\begin{equation}\label{psi}
\psi(q):=\frac{\left(q^2;q^2\right)^2_{\infty}}{\left(q\right)_{\infty}}=\sum_{n\ge 0}q^{\frac{n(n+1)}{2}}.
\end{equation}
Next, we state a result of Ramanujan, which simply follows by splitting into odd and even exponents.
\begin{lemma}\label{thetadissection}
	We have
	\[
	\Theta(q)=\Theta\left(q^4\right)+2q\psi\!\left(q^8\right).
	\]
\end{lemma} 

Due to Andrews \cite[(4.3)]{A}, we have the following Lerch-type representation of\,\footnote{The factor in the denominator in the sum on the right-hand side of \cite[(4.3)]{A} is $(q;q^2)_{n+1}$ (after fixing the typographical error $(q)_{2n+1}$).} $B(q)$. 

\begin{lemma}\label{And}
	We have 
	\[
	B(q)=\frac{\left(-q^2;q^2\right)_{\infty}}{\left(q^2;q^2\right)_{\infty}}\sum_{n\in \mathbb{Z}}\frac{(-1)^n q^{2n(n+1)}}{1-q^{2n+1}}.
	\]
\end{lemma} 

From \cite[(4.1)]{Mao}, we obtain that the restriction of $B(q)$ to exponents divisible by $4$ is a simple modular form.
\begin{lemma}\label{Mao}
	We have
	\[
	\sum_{n\ge 0}c_B(4n)q^n=\frac{\left(q^{2};q^2\right)^{14}_{\infty}}{(q)^9_{\infty} \left(q^4;q^4\right)^4_{\infty}}.
	\]
\end{lemma}

From Chan and Mao \cite[(2.13)]{CM}, we can conclude the following restriction of $B$ to exponents congruent to $1$ modulo $4$.
\begin{lemma}\label{CM}
	We have 
	\[
	\sum_{n\ge 0}c_B(4n+1)q^n
	=\frac{2\left(q^2;q^2\right)^8_{\infty}}{(q)^7_{\infty}}. 
	\]
\end{lemma}

 Next, we require an identity between $\o(q)$ and $f(q)$, where Ramanujan's third order mock theta function $f(q)$ is defined by
\begin{equation*}
f(q):=\sum_{n\ge0}\frac{q^{n^2}}{\left(-q\right)^2_{n}}.
\end{equation*}

Due to Watson \cite[p. 66]{W} (see also \cite[(4)]{APSY}), we have the following identity. 
\begin{lemma}\label{odissection}
We have 
\[
f\!\left(q^8\right)+2q\o(q)+2q^3\o\!\left(-q^4\right)=\frac{\Theta(q)\Theta^2\left(q^2\right)}{\left(q^4;q^4\right)^2_{\infty}}:=G(q).
\]
\end{lemma}

Due to Andrews, Passary, Sellers, and Yee \cite[Lemma 3.1]{APSY}, we have the following dissection of $G(q)$.
\begin{lemma}\label{Fdissection}
We have
\[
G\left(q\right)=\sum_{j=0}^{4} q^jG_j\!\left(q^4\right),
\]
where
\begin{align*}
G_0(q)&:=\frac{\Theta^3(q)}{(q)^2_{\infty}},\hspace{2.6 cm}  G_1(q):=\frac{2\Theta^2(q)\psi\left(q^2\right)}{(q)^2_{\infty}},\\ G_2(q)&:=\frac{4\Theta(q)\psi^2\!\left(q^2\right)}{(q)^2_{\infty}},\hspace{1.5 cm} G_3(q):=\frac{8\psi^3\!\left(q^2\right)}{(q)^2_{\infty}}.
\end{align*}
\end{lemma}

Finally, we state congruences for $c_{\o}(n)$ \cite[Theorem 3.4]{APSY}. 
\begin{lemma}\label{Ref1}
	For $n\in \mathbb{N}_0$, we have 
	\begin{align*}
	c_{\o}\left(8n+3\right)&\equiv 0\!\pmod{4},\\
	c_{\o}\left(8n+5\right)&\equiv 0\!\pmod{8},\\
	c_{\o}\left(16n+12\right)&\equiv 0\!\pmod{4}.
	\end{align*}
\end{lemma}

\section{Proof of \Cref{lem1}}\label{sec2}
 
 In this section, we prove \Cref{lem1}.

\begin{proof}[Proof of \Cref{lem1}]
By \Cref{lemident1} with $q\mapsto q^2$, we have
\begin{align*}
\sum_{n\ge 0}\frac{\left(-aq^2,-bq^2;q^2\right)_nq^{2n+2}}{\left(-cq^2;q^2\right)_n}
&=\frac{c}{ab}\sum_{n\ge 1}\!\frac{\left(-c^{-1};q^2\right)_n \pa{ab}{c}^n q^{n(n+1)}}{\left(\frac{aq^2}{c},\frac{bq^2}{c};q^2\right)_n}\\
&\hspace{1 cm}-\frac{c\left(-aq^2,-bq^2;q^2\right)_{\infty}}{ab\left(-cq^2;q^2\right)_{\infty}}\sum_{n\ge 1}\!\frac{\pa{ab}{c^2}^nq^{2n^2}}{\left(\frac{aq^2}{c},\frac{bq^2}{c};q^2\right)_n}.
\end{align*}
Plugging $a=b=-q^{-1}$ and $c=1$ into the above, we obtain 
\begin{align*}
\sum_{n\ge 0}\frac{\left(q;q^2\right)^2_nq^{2n+2}}{\left(-q^2;q^2\right)_n}
&=2q^2\sum_{n\ge 1}\!\frac{\left(-q^2;q^2\right)_{n-1} q^{n(n-1)}}{\left(-q;q^2\right)^2_n}-\frac{\left(q;q^2\right)^2_{\infty}q^2}{\left(-q^2;q^2\right)_{\infty}}\sum_{n\ge 1}\!\frac{q^{2n(n-1)}}{\left(-q;q^2\right)^2_n}\\
&=2q^2\sum_{n\ge 0}\!\frac{\left(-q^2;q^2\right)_{n} q^{n(n+1)}}{\left(-q;q^2\right)^2_{n+1}}-\frac{\left(q;q^2\right)^2_{\infty}q^2}{\left(-q^2;q^2\right)_{\infty}}\sum_{n\ge 0}\!\frac{q^{2n(n+1)}}{\left(-q;q^2\right)^2_{n+1}}\\
&=2q^2B(-q)-\frac{\left(q;q^2\right)^2_{\infty}q^2}{\left(-q^2;q^2\right)_{\infty}}\o(-q),
\end{align*}
using \eqref{McIntosh} and \eqref{Ramanujanomega} with $q\mapsto -q$. Multiplying both sides by $q^{-2}$ and using \eqref{S}, the theorem follows.\qedhere 
\end{proof}


\section{Proof of \Cref{ABConj1}}\label{sec3}

 In this section, we prove \Cref{ABConj1}.

\begin{proof}[Proof of \Cref{ABConj1}]
Using \eqref{C}, \eqref{A}, and \Cref{lem1}, we have 
\begin{align}\label{eqn1}
C(q)
&=\frac{2q\left(-q^2;q^2\right)_{\infty}}{\left(q;q^2\right)^2_{\infty}}B(-q)-q\o(-q).
\end{align}
By \Cref{Ref1}, we have, for $n\in \mathbb{N}_0$, 
\begin{equation}\label{ref1}
\textnormal{coeff}_{\left[q^{8n+3}\right]}\left(\o(q)\right)\equiv 0\pmod{4}.
\end{equation}
Now, for a $q$-series $H(q)=\sum_{n\ge 0}h(n)q^n$ and $n\in \mathbb{N}_0$, we define
\[
\operatorname{coeff}_{\left[q^n\right]}\left(H(q)\right):=h(n).
\]
Note that \eqref{ref1} is equivalent to \[\textnormal{coeff}_{\left[q^{8n+4}\right]}\left(q\o(-q)\right)\equiv 0\pmod{4}.\]
 Thus, to prove \Cref{ABConj1}, by \eqref{A} and \eqref{eqn1}, we need to show that, for $n\in \mathbb{N}_0$,
\begin{align*}
 \textnormal{coeff}_{\left[q^{8n+3}\right]}\!\left(A(q)B(-q)\right)\!\equiv\! 0\!\pmod{2}.
\end{align*}

 Now, we reduce $A(q)B(-q)$ modulo $2$. First, we obtain, using (1.2), 
 \begin{align*}
 A(q)=\frac{\left(-q^2;q^2\right)_{\infty}}{\left(q;q^2\right)^2_{\infty}}
 &\!\equiv\!\left(q^2;q^2\right)^2_{\infty}\!\equiv\! \left(q^4;q^4\right)_{\infty}\!\pmod{2}.
 \end{align*}
 The final $q$-series only has exponents divisible by $4$.

 Next, we consider $B(-q)$. By \Cref{And}, we have 
 \begin{equation*}
 B(-q)\equiv B(q)\equiv\sum_{n\in \mathbb{Z}}\frac{(-1)^n q^{2n(n+1)}}{1-q^{2n+1}}\pmod{2},
 \end{equation*}
using the fact that $\frac{(-q^2;q^2)_{\infty}}{(q^2;q^2)_{\infty}}\equiv 1\pmod{2}$. We then rewrite 
\begin{align*}
\sum_{n\in \mathbb{Z}}\frac{(-1)^n q^{2n(n+1)}}{1-q^{2n+1}}&\equiv \sum_{n\in \mathbb{Z}}\frac{ q^{2n(n+1)}}{1+q^{2n+1}}=\frac 12\sum_{n\in \mathbb{Z}}\left(\frac{q^{2n(n+1)}}{1+q^{2n+1}}+\frac{q^{2n(n+1)}}{1+q^{-2n-1}}\right)\nonumber\\
&=\frac 12\sum_{n\in \mathbb{Z}}q^{2n(n+1)}\pmod{2}.
\end{align*}
Again, note that the final $q$-series only has exponents divisible by $4$. Combining gives the claim. 
This concludes the proof of \Cref{ABConj1}.\qedhere 
\end{proof}


\section{Proof of \Cref{ABConj1a}}\label{sec3a}

In this section, we prove \Cref{ABConj1a}. 

\begin{proof}[Proof of \Cref{ABConj1a}]
We split the proof into two parts corresponding to the two terms in \Cref{lem1} (after multiplying the pre-factor, see \eqref{C}). The contribution of the mock theta function is handled using known congruences, while the remaining term is treated via a $4$-dissection and cancellation argument.

 By \Cref{Ref1}, we have, for $n\in \mathbb{N}_0$, 
\[\textnormal{coeff}_{\bigl[q^{8n+6}\bigr]}\left(q\o(-q)\right)\equiv 0\pmod{8}.\]
Thus, to prove \Cref{ABConj1a}, by \eqref{eqn1}, \eqref{A}, and \Cref{Ref1}, we need to show that, for $n\in \mathbb{N}_0$, 
\begin{align*}
 \textnormal{coeff}_{\bigl[q^{8n+5}\bigr]}\!\left(A(q)B(-q)\right)\!\equiv\! 0\!\pmod{4}.
\end{align*}
In the following we even prove the stronger result that 
\begin{equation}\label{TSalt}
\textnormal{coeff}_{\left[q^{4n+1}\right]}\left(A(q)B(-q)\right)\equiv 0\pmod{4}.
\end{equation}

We first show that\footnote{Andrews and Bachraoui \cite[Remark 3]{AB} noticed that \eqref{claim1} holds. However, we give a proof for the reader's convenience.} 
 \begin{equation}\label{claim1}
\textnormal{coeff}_{\left[q^{4n+j}\right]}\left(A(q)\right)\equiv 0\pmod{4}\ \ \text{for}\ \ j\in\{2,3\}.
 \end{equation}
For this definitions \eqref{A} and \eqref{thetas}, we may write 
\begin{equation}\label{claim1eqn1}
A(q)
=\frac{\left(q^4;q^4\right)_{\infty}}{\Theta(-q)}.
\end{equation}
Now, by \eqref{thetas}, note that 
\[
\Theta(-q)\equiv 1+2T(q)\pmod{4},
\]
where $T(q):=\sum_{n\ge 1} q^{n^2}$. Plugging this into \eqref{claim1eqn1}, we obtain 
\begin{align}\label{claim1eqn2}
\hspace{-0.33 cm}A(q)&\equiv\!\frac{\left(q^4;q^4\right)_{\infty}}{1+2T(q)}\!\equiv\! \left(q^4;q^4\right)_{\infty}\!\left(1-2T(q)\right)\!\equiv\! \left(q^4;q^4\right)_{\infty}\!\left(1+2T(q)\right)\!\!\pmod{4}.
\end{align}
Next, we write 
\[
T(q)=
\sum_{n\ge 1}q^{4n^2}+q\sum_{n\ge 0}q^{4n(n+1)}.
\]
Plugging this into \eqref{claim1eqn2}, we obtain
\begin{equation}\label{claim1eqn3}
A(q)\equiv \left(q^4;q^4\right)_{\infty}\sum_{n\in \mathbb{Z}}q^{4n^2}+2q\left(q^4;q^4\right)_{\infty}\sum_{n\ge 0}q^{4n(n+1)}\pmod{4}.
\end{equation}
This proves \eqref{claim1}. 

Next, by \eqref{claim1eqn3}, we have 
\begin{equation*}
A(q)\equiv A_0\!\left(q^4\right)+qA_1\!\left(q^4\right)\pmod{4},
\end{equation*}
where 
\begin{align}\label{claim2eqn1}
A_0(q)&:=\left(q\right)_{\infty}\sum_{n\in \mathbb{Z}}q^{n^2}=\frac{\left(q^2;q^2\right)^5_{\infty}}{\left(q\right)_{\infty}\left(q^{4};q^{4}\right)^2_{\infty}},\nonumber\\
A_1(q)&:=2\left(q\right)_{\infty}\sum_{n\ge 0}q^{n(n+1)}=\frac{2\left(q\right)_{\infty}\left(q^{4};q^{4}\right)^2_{\infty}}{\left(q^2;q^2\right)_{\infty}},
\end{align}
using \eqref{theta} and \eqref{psi}. 

Now write
\[
B(q)=\sum_{j=0}^{3}q^j\!B_j\!\left(q^4\right).
\]
Then 
\[
A(q)B(-q)\equiv \left(A_0\!\left(q^4\right)+qA_1\!\left(q^4\right)\right)\sum_{j=0}^{3}(-1)^jq^j\!B_j\!\left(q^4\right)\!\pmod{4}=:\sum_{j=0}^3q^j\!\mathcal{F}_j\!\left(q^4\right),
\]
for certain $q$-series $\mathcal{F}_j$. In particular, 
\[
\mathcal{F}_1(q)=-A_0\!\left(q\right)B_1\!\left(q\right)+A_1\!\left(q\right)B_0\!\left(q\right).
\]
To finish the proof of \eqref{TSalt}, we need to show that
\begin{equation}\label{F1}
\mathcal{F}_1(q)\equiv 0\pmod{4}.
\end{equation}
We now determine $B_0$ and $B_1$. For this, using \Cref{Mao} and \Cref{CM}, we note 
\begin{align}\label{claim2eqn2}
B_0(q)&=\sum_{n\ge 0}c_B(4n)q^n=\frac{\left(q^{2};q^2\right)^{14}_{\infty}}{\left(q\right)^9_{\infty} \left(q^{4};q^{4}\right)^4_{\infty}},\nonumber\\
B_1(q)&=\sum_{n\ge 0}c_B(4n+1)q^n=-\frac{2\left(q^2;q^2\right)^8_{\infty}}{\left(q\right)^7_{\infty}}.
\end{align}
 Combining \eqref{claim2eqn1} and \eqref{claim2eqn2}, we obtain 
 \begin{align*}
 A_0(q)B_1(q)
 =-\frac{2\left(q^{2};q^2\right)^{13}_{\infty}}{\left(q\right)^8_{\infty} \left(q^{4};q^{4}\right)^2_{\infty}}.
 \end{align*}
Similarly, combining \eqref{claim2eqn1} and \eqref{claim2eqn2}, we obtain 
\begin{align*}
A_1(q)B_0(q)
=\frac{2\left(q^{2};q^2\right)^{13}_{\infty}}{\left(q\right)^8_{\infty} \left(q^{4};q^{4}\right)^2_{\infty}}.\\
\end{align*}
Thus \eqref{F1} holds.\qedhere 
\end{proof}


\section{Proof of \Cref{newthm}}\label{sec3b}

In this section, we prove \Cref{newthm}. First, we show the following.

\begin{proposition}\label{newprop}
	We have
	\[
	A(q)B(-q)\equiv (q^{16};q^{16})_\infty \pmod 2.
	\]
\end{proposition}

\begin{proof}
	First, by \eqref{A}, we have
	\[
	A(q)=\frac{(-q^2;q^2)_\infty}{(q;q^2)_\infty^2}
	\equiv \frac{(q^2;q^2)_\infty}{(q^2;q^4)_\infty}
	=(q^4;q^4)_\infty \pmod 2.
	\]
	Next, \Cref{And} gives
	\[
	B(-q)=\frac{(-q^2;q^2)_\infty}{(q^2;q^2)_\infty}
	\sum_{n\in \mathbb{Z}}\frac{(-1)^n q^{2n(n+1)}}{1+q^{2n+1}}
	\equiv
	\sum_{n\in \mathbb{Z}}\frac{q^{2n(n+1)}}{1+q^{2n+1}} \pmod 2.
	\]
	Pairing the terms for $n$ and $-n-1$, we obtain
	\[
	\frac{q^{2n(n+1)}}{1+q^{2n+1}}
	+
	\frac{q^{2n(n+1)}}{1+q^{-2n-1}}
	= q^{2n(n+1)}.
	\]
	Hence
	\[
	B(-q)\equiv \sum_{n\ge 0} q^{2n(n+1)} = \psi(q^4) \pmod 2.
	\]
	Using \eqref{psi}, we conclude that
	\[
	A(q)B(-q)
	\equiv \left(q^4;q^4\right)_\infty\,\psi\!\left(q^4\right)
	= \left(q^8;q^8\right)_\infty^2
	\equiv \left(q^{16};q^{16}\right)_\infty \pmod 2.\qedhere
	\]
\end{proof}

Now we are ready to prove \Cref{newthm}.

\begin{proof}[Proof of \Cref{newthm}]
By \Cref{newprop}, we have
\begin{equation}\label{neqn1}
2qA(q)B(-q)\equiv 2q\left(q^{16};q^{16}\right)_\infty \pmod 4.
\end{equation}
 Since $\textnormal{coeff}_{[q^{16n+j}]}(q(q^{16};q^{16})_\infty)=0$ for $j\neq 1$, by \eqref{neqn1},
$\textnormal{coeff}_{[q^{16n+13}]}(2qA(q)$\ $B(-q))\equiv 0\pmod{4}$. On the other hand, 
$\textnormal{coeff}_{[q^{16n+13}]}(q\omega(-q))=c_{\omega}(16n+12)$, because $(-1)^{16n+12}=1$.
By \Cref{Ref1}, we obtain
\[
c_{\omega}(16n+12)\equiv 0 \pmod 4.
\]
Therefore, using \eqref{eqn1}, we conclude the claim.\qedhere
\end{proof}

\section{Proof of \Cref{ABConj2}}\label{sec4}

In this section, we prove \Cref{ABConj2}. Write 
\[
\Theta(q)\o(q)=:\sum_{n\ge 0}\d(n)q^n.
\]
 For a $q$-series $H(q):=\sum_{n\ge 0}h(n) q^n$, we define the operator 
\[
\mathcal{R}\left(H(q)\right):=\sum_{n\ge 0}h\!\left(4n+1\right) q^n.
\] and set 
\[
M(q):=\mathcal{R}\left(\Theta(q)\o(q)\right).
\]

In the following proposition we show that $M$ is a modular form and determine its explicit shape.

\begin{proposition}\label{eq:main-id}
	We have 
	\begin{equation*}
	M(q)
	=
	\frac{4\left(q^2;q^2\right)^2_\infty}{\left(q;q^2\right)^6_\infty}.
	\end{equation*}
\end{proposition}

\begin{proof}
 

We write $\o(q)$ as 
	\[
	\o(q)=\sum_{j=0}^{3}q^j\o_{j}\!\left(q^4\right). 
	\]
\Cref{thetadissection} gives 
\begin{equation}\label{ref2}
M(q)=\Theta(q)\o_1(q)+2\psi\!\left(q^2\right)\o_0(q).
\end{equation}
 By \Cref{odissection}, we have
\[
\o(q)=\frac 12q^{-1}G(q)-\frac 12q^{-1}f\!\left(q^8\right)-\frac 12 q^2\o\!\left(-q^4\right).
\]
Note that only the first term contributes to $\o_0$ and $\o_1$. By \Cref{Fdissection}, we obtain
\[
\frac 12q^{-1}G(q)=\frac 12\sum_{j=0}^3q^{j-1}G_j\!\left(q^4\right).
\]
Then, Lemma 2.7 yields
\begin{align}\label{eqn2}
\o_0(q)&=\frac{G_1(q)}2=\frac{\Theta^2(q)\psi\!\left(q^2\right)}{(q)^2_{\infty}},\quad  \o_1(q)=\frac{G_2(q)}2=\frac{2\Theta(q)\psi^2\!\left(q^2\right)}{(q)^2_{\infty}}.
\end{align}

By \eqref{ref2}, we then note that
\begin{align*}
&M\!\left(q\right)=
\frac{4\Theta^2\!\left(q\right)\psi^2\!\left(q^2\right)}{\left(q\right)^2_\infty}
=\frac{4\left(q^2;q^2\right)^{2}_\infty}{\left(q;q^2\right)^6_\infty},
\end{align*}
using \eqref{eqn2}, \eqref{psi}, and \eqref{theta}. This gives the claim.\qedhere 
\end{proof}

We are now ready to show \Cref{ABConj2}.

\begin{proof}[Proof of \Cref{ABConj2}]
Set $D(q):=\frac{(q;q^2)^2_{\infty}}{(-q^2;q^2)_{\infty}}\o(-q)$. Now, by \Cref{lem1},
\begin{equation}\label{B1}
\mathcal{R}\left(S(q)\right)=\mathcal{R}\left(2B(-q)\right)-\mathcal{R}\left(D(q)\right).
\end{equation}
By \Cref{CM}, we have 
\begin{equation}\label{B2}
\mathcal{R}\left(2B(-q)\right)=-2\sum_{n\ge 0}c_B(4n+1)q^n=-\frac{4\left(q^2;q^2\right)^8_{\infty}}{\left(q\right)^7_{\infty}}=-\frac{4\left(q^2;q^2\right)_{\infty}}{\left(q;q^2\right)^7_{\infty}}.
\end{equation}
Next, by \eqref{thetas}, we obtain 
\[
D(q)=\frac{\Theta(-q)\o(-q)}{\left(q^4;q^4\right)_{\infty}}.
\]
 Since, $\Theta(-q)\o(-q)=\sum_{n\ge 0}(-1)^n\d(n)q^n$, we have
\[
\mathcal{R}\left(\Theta(-q)\o(-q)\right)=-M(q).
\]
 As $(q^4;q^4)^{-1}$ contains only powers of $q^{4m}$ ($m\in \mathbb{N}_{ 0}$), we obtain
\[
\mathcal{R}\left(D(q)\right)=-\frac{M(q)}{\left(q\right)_{\infty}}.
\]
By \Cref{eq:main-id}, we have
\[
M(q)=\frac{4\left(q^2;q^2\right)^2_\infty}{\left(q;q^2\right)^6_\infty},
\]
so
\begin{equation}\label{B3}
\mathcal{R}\left(D(q)\right)=-\frac{4\left(q^2;q^2\right)^2_\infty}{\left(q\right)_{\infty}\left(q;q^2\right)^6_\infty}=-\frac{4\left(q^2;q^2\right)_{\infty}}{\left(q;q^2\right)^7_{\infty}}.
\end{equation}
The right-hand sides of \eqref{B2} and \eqref{B3} are equal, hence
\[
\mathcal{R}\left(2B(-q)\right)=\mathcal{R}\left(D(q)\right).
\]
Thus, by \eqref{B1}, we conclude the proof.\qedhere

\end{proof}

\section{Open questions}\label{sec5}

  We conclude this paper by proposing following conjectures related to congruences of $c(n)$ for future research.

\begin{conjecture}\label{conj1}
	For every $n\in \mathbb{N}_0$, we have\footnote{We have checked this numerically for $32n+23\le 5000$.}
	\[
	c(32n+23)\equiv 0 \pmod 8.
	\]
\end{conjecture}

 More generally, we propose the following infinite family of congruences for $c(n)$ modulo $4$ and $8$. 

\begin{conjecture}\label{conj2}
For $k,n\in \mathbb{N}_0$, we have
\begin{align*}
c\left(2^{2k+3}n+\frac{11\cdot 4^k+1}{3}\right)\equiv 0\pmod{4},\\
c\left(2^{2k+3}n+\frac{17\cdot 4^k+1}{3}\right)\equiv 0\pmod{8},\\
c\left(2^{2k+4}n+\frac{38\cdot 4^k+1}{3}\right)\equiv 0\pmod{4}.
\end{align*}	
\end{conjecture}

\begin{remark}
{\it Note that we settle the case $k=0$ of \Cref{conj2} in Theorems \ref{ABConj1}--\ref{newthm}. The case $k=1$ (see the congruence modulo $8$) is \Cref{conj1}.}
\end{remark}


\end{document}